\documentclass[12pt]{amsart}
\usepackage{hyperref,amsmath,amsthm,amssymb,amscd,epsfig,latexsym,graphicx,eucal,layout,fancyhdr,bm,mathrsfs}

\newtheorem{theorem}{Theorem}

\theoremstyle{definition}

\newcommand{\con}{\operatorname{const.}}

\newcommand{\fix}{\operatorname{fix}}
\newcommand{\isom}{\operatorname{Isom}^+(\HH^{n})}
\newcommand{\mob}{\operatorname{M\ddot{o}b}}

\newcommand{\bb}{\mathscr{B}}

\newcommand{\bd}{\partial}

\newcommand{\es}{\emptyset}
\newcommand{\ov}{\overline}
\newcommand{\sm}{\smallsetminus}
\newcommand{\ve}{\varepsilon}

\newcommand{\vs}{\vspace{2mm}}

\newcommand{\RR}{{\mathbb R}}
\newcommand{\HH}{{\mathbb H}}

\newcommand{\NN}{{\mathbb N}}

\newcommand{\bit}{\it \bfseries}

\newcommand{\thmref}[1]{Theorem~\ref{#1}}

\newcommand{\figref}[1]{Fig.~\ref{#1}}

\input{psbox}

\psfordvips

\begin{document}

\title[A discreteness criterion]{A discreteness criterion for groups containing parabolic isometries}

\author[V. Erlandsson and S. Zakeri]{Viveka Erlandsson and Saeed Zakeri}

\address{V. Erlandsson, Department of Mathematics, Graduate Center of
CUNY, New York}

\email{verlandsson@gc.cuny.com}

\address{S. Zakeri, Department of Mathematics, Queens College and Graduate Center of
CUNY, New York}

\email{saeed.zakeri@qc.cuny.edu}



\date{July 23, 2013}

\begin{abstract}
This note will prove a discreteness criterion for groups of orientation-preserving isometries of the hyperbolic space which contain a parabolic element. It can be viewed as a generalization of the well-known results of Shimizu-Leutbecher and J{\o}rgensen in dimensions $2$ and $3$, and is closely related to Waterman's inequality in higher dimensions. Unlike his algebraic method, the argument presented here is geometric and yields an improved asymptotic bound. 
\end{abstract}

\maketitle

\section{Preliminaries}\label{sec:pre}

We will need a few basic facts about the hyperbolic space and its isometries, as well as the notion of the Margulis region associated with parabolic fixed points. Most of this material is standard and can be found, for example, in \cite{BP}, \cite{R}, and \cite{Th}. \vs

Throughout we will use the upper half-space model for the $n$-dimensional real hyperbolic space:
$$
\HH^{n} = \{ x = (v,t) : v \in \RR^{n-1}, t>0 \}.
$$
The extended boundary $\bd \HH^{n} \cong \ov{\RR} \, \! ^{n-1} =\RR^{n-1} \cup \{ \infty \}$ is homeomorphic to the sphere of dimension $n-1$, with the closure $\ov{\HH} \, \! ^{n}= \HH^{n} \cup \ov{\RR} \, \! ^{n-1}$ homeomorphic to the closed $n$-ball. The hyperbolic metric $(dv^2+dt^2)/t^2$ on $\HH^{n}$ induces the distance $\rho(\cdot,\cdot)$ which satisfies
\begin{equation}\label{dist}
\cosh(\rho(x,\hat{x}))=1+\frac{\| x-\hat{x} \|^2}{2 t \hat{t}} \qquad (x,\hat{x} \in \HH^{n}).
\end{equation}
Here $\| \cdot \|$ is the Euclidean norm in $\RR^{n}$ and $t,\hat{t}$ are the heights of $x,\hat{x}$. \vs

We denote by $\isom$ the group of orientation-preserving isometries of $\HH^{n}$ with respect to hyperbolic metric. Every element of $\isom$ extends continuously to a M\"{o}bius map acting on $\ov{\RR} \, \! ^{n-1}$. Conversely, the Poincar\'e extension of every M\"{o}bius map of $\ov{\RR} \, \! ^{n-1}$ is an element of $\isom$. It follows that $\isom$ is canonically isomorphic to the group $\mob(n-1)$ of orientation-preserving M\"{o}bius maps acting on the $(n-1)$-sphere. For each $g \in \isom$, the fixed point set $\fix(g)=\{ x \in \ov{\HH} \, \! ^{n} : g(x)=x \}$ is non-empty. A non-identity $g$ is {\bit elliptic} if $\fix(g)$ intersects $\HH^{n}$, {\bit loxodromic} if $\fix(g)$ consists of two distinct points on $\ov{\RR} \, \! ^{n-1}$, and {\bit parabolic} if $\fix(g)$ consists of a unique point on $\ov{\RR} \, \! ^{n-1}$. In the last case, one can always assume $\fix(g)= \{ \infty \}$ after a conjugation. The parabolic isometry $g$ will then have the normal form
\begin{equation}\label{normalform1}
g: (v,t) \mapsto (Av+a,t) \qquad (v \in \RR^{n-1},t>0),
\end{equation}
where $A \in \text{SO}(n-1)$ and $a \in \RR^{n-1}$ is non-zero. The conjugacy class of $A$ in $\text{SO}(n-1)$ is uniquely determined by $g$. We call $g$ a {\bit pure translation} if $A=I$, and a {\bit screw translation} if $A \neq I$. A screw translation is {\bit rational} if $A$ has finite order, and is {\bit irrational} otherwise. When $n=4$, these correspond to the cases where the angle of rotation of $A \in \text{SO}(3)$ about its axis is a rational or irrational multiple of $2\pi$. Note that screw translations can exist only when $n \geq 4$. This is because for every non-identity $A \in \text{SO}(2)$ the matrix $A-I$ is invertible, so the map $v \mapsto Av+a$ has a finite fixed point at $-(A-I)^{-1}a$. \vs

Let $\Gamma \subset \isom$ be a discrete group containing a parabolic element which fixes $\infty$. Denote by $\Gamma_{\infty}$ the stabilizer subgroup $\{ g \in \Gamma: g(\infty)=\infty \}$. Since $\Gamma_{\infty}$ is discrete, it cannot contain loxodromic elements \cite[Lemma D.3.6]{BP}, so every element of $\Gamma_{\infty}$ is either parabolic or a finite-order elliptic. For a fixed $\ve>0$, define the {\bit Margulis region} associated with $\Gamma_{\infty}$ by
$$
T = \{ x \in \HH^{n} : \rho(g(x),x)<\ve \ \text{for some parabolic} \ g \in \Gamma_{\infty} \}.
$$
It is easy to see that $T$ is a non-empty domain having $\infty$ on its boundary. Furthermore, when $\ve$ is small enough (less than the Margulis constant of $\HH^{n}$), $T$ is {\bit precisely invariant} under the action of $\Gamma_{\infty}$ in sense that  
$$
\begin{cases}
h(T)=T & \qquad \text{if} \ \ h \in \Gamma_\infty \\
h(T) \cap T =\es & \qquad \text{if} \ \ h \in \Gamma \sm \Gamma_\infty
\end{cases}
$$
(see \cite[Lemma 2.3]{EZ}). The explicit description of $T$ is quite intricate and depends both on the algebraic structure of $\Gamma_{\infty}$ and the Diophantine properties of the corresponding rotation subgroup of $\text{SO}(n-1)$. Fortunately, it will be sufficient for our purposes to work with a subdomain of $T$ which is easier to describe. Take a parabolic element $g \in \Gamma_{\infty}$ and consider the domain
$$
T_g = \{ x \in \HH^{n} : \rho(g^i(x),x)<\ve \ \text{for some} \ i \in \NN \}.
$$
Clearly, $T_g \subset T$ so we still have the property 
\begin{equation}\label{disj}
h(T_g) \cap T_g = \es \qquad \text{if} \ h \in \Gamma \sm \Gamma_\infty.
\end{equation}
By the distance formula \eqref{dist}, the condition $\rho(g^i(x),x)<\ve$ on $x=(v,t) \in \HH^{n}$ is equivalent to  
$$
\frac{\| g^i(v)-v \|^2}{2t^2} < \cosh(\ve)-1
$$
or
$$
t> c(\ve) \ \| g^i(v)-v \|, \quad \text{where} \quad c(\ve) =\frac{1}{\sqrt{2\cosh(\ve)-2}}. 
$$
If we define the sequence of functions $\{ u_{g,i}: \RR^{n-1} \to \RR \}_{i \in \NN}$ by
\begin{equation}\label{ui}
u_{g,i}(v)= c(\ve) \ \| g^i(v)-v \| 
\end{equation}
and the {\bit boundary function} $\bb_g : \RR^{n-1} \to \RR$ by
\begin{equation}\label{bb}
\bb_g(v)= \inf_{i \in \NN} \, u_{g,i}(v),
\end{equation}
it follows that
\begin{equation}\label{tg}
T_g = \{ (v,t) \in \HH^{n} :  t > \bb_g(v) \}.
\end{equation}

The geometry of $T_g$ is of interest to us only when $g$ is an irrational screw translation since otherwise $\bb_g$ is a bounded function and $T_g$ contains a precisely invariant horoball based at $\infty$. To see this, consider the normal form \eqref{normalform1} for $g$. After conjugating by a translation, we may assume $Aa=a$. Consider the $A$-invariant orthogonal decomposition 
$$
\RR^{n-1} = E \oplus E^\perp,
$$
where $E=\{ v \in \RR^{n-1} : Av=v \}$. Splitting $v \in \RR^{n-1}$ accordingly as $v=w+w^\perp$, we see that 
$$
g^i(v)-v = A^i v+ i a - v = (A^i-I) w^\perp + ia,
$$
hence
$$
\| g^i(v)-v \|^2 = \|(A^i-I)w^\perp \|^2 + i^2 \| a \|^2
$$
since $ia \in E$ is orthogonal to $(A^i-I) w^\perp \in E^\perp$. It follows that 
\begin{equation}\label{pyth}
u_{g,i}(v)= c(\ve) \ \left( \|(A^i-I)w^\perp \|^2 + i^2 \| a \|^2 \right)^{1/2}.
\end{equation}
Now if $g$ is a pure or rational screw translation, \eqref{bb} and \eqref{pyth} imply that 
\begin{equation}\label{prat}
\bb_g(v) \leq c(\ve) \, i \| a \|, \quad \text{where} \ i \in \NN \  \text{is the order of} \ A 
\end{equation}
(equality holds for all $v$ if $i=1$, and for all $v$ with large $\| w^{\perp} \|$ if $i>1$). In this case, the description \eqref{tg} shows that $T_g$, hence the Margulis region $T$, contains the horoball 
$$
\{ (v,t) \in \HH^{n}: t > c(\ve) \, i \| a \| \}  
$$
based at $\infty$, which is easily seen to be precisely invariant under the action of $\Gamma_\infty$. 

\section{The Criterion}\label{sec:th}

We will use the geometry of the Margulis region to prove a discreteness criterion for subgroups of $\isom$ containing parabolic isometries. Our result can be viewed as a generalization of the well-known results of Shimizu-Leutbecher \cite{Sh} and J{\o}rgensen \cite{J} in dimensions $2$ and $3$ which give a uniform bound on the radii of the isometric spheres of elements in a discrete group which move the parabolic fixed point $\infty$. In higher dimensions, Ohtake \cite{Oh} has shown that no such uniform bound exists in the presence of an irrational screw translation. In \cite{W}, Waterman obtains a bound on the radii of the isometric spheres as a function of the Euclidean distance that their centers move under the parabolic isometry (see \eqref{Wat}). Our result similarly bounds the radii in terms of the centers of the isometric spheres. \vs

Let us first recall some basic facts about isometric spheres. Consider a map $h \in \mob(n-1)$ with $h(\infty) \neq \infty$, and let $v=h^{-1}(\infty)$, $v'=h(\infty)$. It is not hard to check that $h$ maps each Euclidean sphere $S(v,R)= \{ w \in \RR^{n-1} : \| w-v \| = R \}$ to a Euclidean sphere $S(v',R')$. Here $R \mapsto R'$ is a decreasing function, with $\lim_{R \to 0} R'=\infty$ and $\lim_{R \to \infty} R'=0$. Hence there is a unique radius $R_h>0$ such that $h(S(v,R_h))=S(v',R_h)$. We call $S(v,R_h)$ the {\bit isometric sphere} of $h$ and denote it by $S_h$. The map $h: S_h \to S_{h^{-1}}$ is a Euclidean isometry, and $R_h=R_{h^{-1}}$. Moreover, $h$ maps the interior (resp. exterior) of $S_h$ to the exterior (resp. interior) of $S_{h^{-1}}$. \vs

If $h \in \isom$ with $h(\infty) \neq \infty$, we consider its boundary map $\hat{h} \in \mob(n-1)$ and define the isometric sphere $S_h$ as the convex hull of the isometric sphere of $\hat{h}$, that is, the upper hemisphere in $\HH^{n}$ with the same center and radius as those of $S_{\hat{h}}$. The center of $S_h$ lies on $\RR^{n-1}$, so it has coordinates $(v_h,0) \in \RR^{n-1} \times \{ 0 \}$. The north pole $N_h$ of $S_h$ has coordinates $(v_h,R_h)$. Since $h(v_h,0)=\infty$ and $h(\infty)=(v_{h^{-1}},0)$, $h$ carries the vertical geodesic through $(v_h,0)$ isometrically to the vertical geodesic through $(v_{h^{-1}},0)$. It follows in particular that $h(N_h)=N_{h^{-1}}$.

\begin{theorem}[Discreteness Criterion]\label{ds}
Suppose $g,h \in \isom$, $g$ is parabolic with $g(\infty)=\infty$, and $h(\infty) \neq \infty$. If the group generated by $g,h$ is discrete, then
\begin{equation}\label{bbb}
R_h \leq \big( \bb_g(v_h) \ \bb_g(v_{h^{-1}}) \big) ^{1/2}.
\end{equation}
\end{theorem}

\begin{proof}
For convenience we drop the subscript $h$ from our notations and denote the quantities associated with $h^{-1}$ with a minus sign. There is nothing to prove if $R \leq \bb_g(v)$ and $R \leq \bb_g(v^-)$, so let us assume that $R>\bb_g(v)$ (the case $R>\bb_g(v^-)$ is similar). The vertical geodesic segment $\sigma$ from the north pole $N=(v,R)$ down to $(v,\bb_g(v))$ lies in $T_g$ and has hyperbolic length $\log(R/\bb_g(v))$ (see \figref{RAD}). The image $h(\sigma)$ is the vertical geodesic segment from $N^-=(v^-,R)$ up to $(v^-,t)$ for some $t$, and has the same length as $\sigma$. It follows that $\log(t/R)=\log(R/\bb_g(v))$, or $R^2=\bb_g(v) t$. Let $\Gamma$ be the group generated by $g,h$, and $\Gamma_\infty$ be the stabilizer of $\infty$ in $\Gamma$. Since $h \in \Gamma \sm \Gamma_{\infty}$, we have $h(T_g) \cap T_g = \es$ by \eqref{disj}. In particular, the image $h(\sigma)$ must be disjoint from $T_g$. This shows $t<\bb_g(v^-)$, and proves the required inequality $R^2 <\bb_g(v) \bb_g(v^-)$. 
\end{proof}

\begin{figure*}
  \includegraphics[width=11cm]{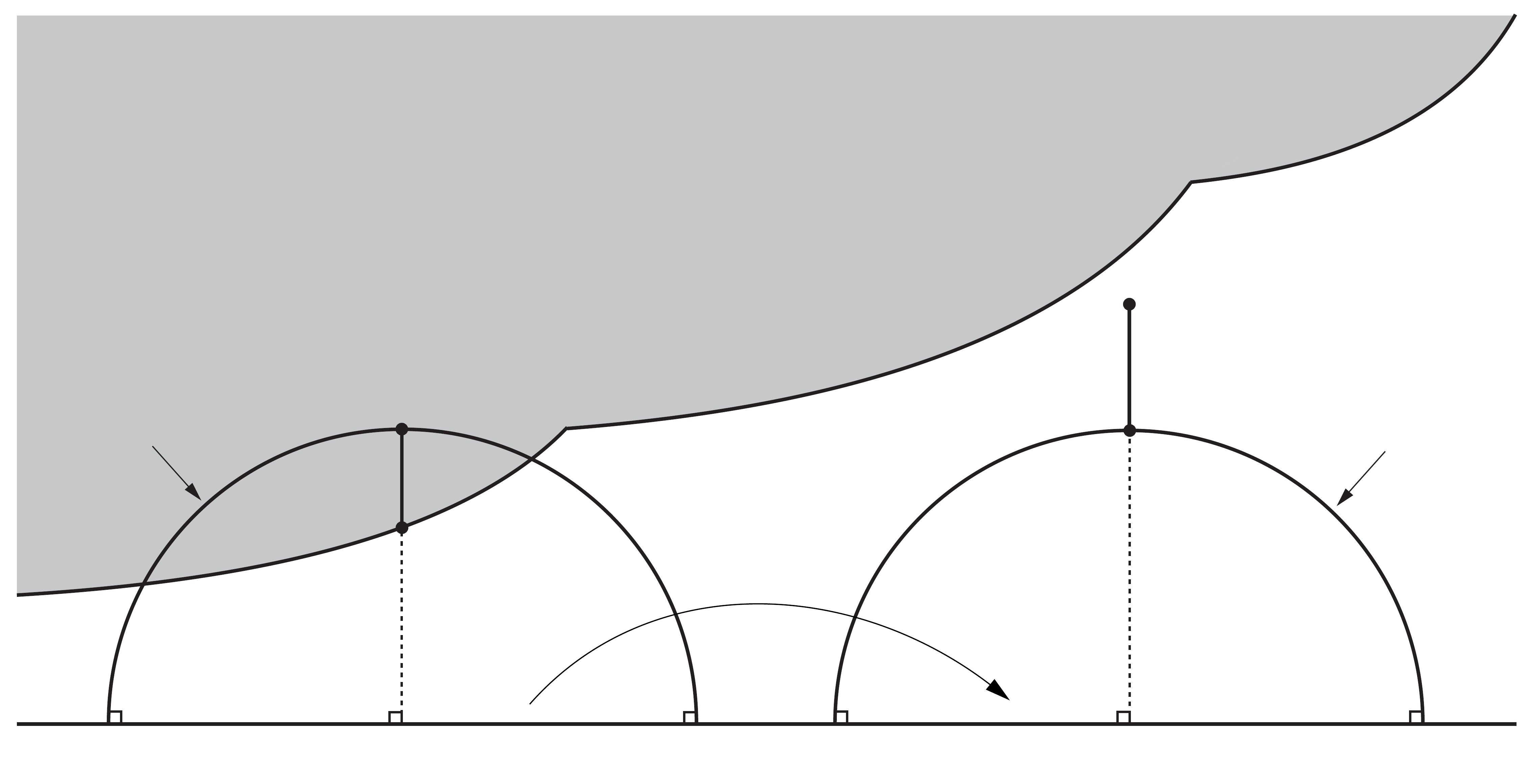}
\caption{{\sl Proof of \thmref{ds}}
\at{-7.0\pscm}{3.4\pscm}{{\footnotesize $S_h$}}
\at{2.2\pscm}{3.4\pscm}{{\footnotesize $S_{h^{-1}}$}}
\at{-2.5\pscm}{5.0\pscm}{{\small $T_g$}}
\at{-5.6\pscm}{1.0\pscm}{{\footnotesize $(v,0)$}}
\at{-5.3\pscm}{2.5\pscm}{{\footnotesize $(v,\bb_g(v))$}}
\at{-5.75\pscm}{3.0\pscm}{{\footnotesize $\sigma$}}
\at{-5.7\pscm}{3.6\pscm}{{\footnotesize $N$}}
\at{-1.0\pscm}{1.0\pscm}{{\footnotesize $(v^-,0)$}}
\at{-1.4\pscm}{3.8\pscm}{{\footnotesize $h(\sigma)$}}
\at{-0.75\pscm}{3.1\pscm}{{\footnotesize $N^-$}}
\at{-0.8\pscm}{4.3\pscm}{{\footnotesize $(v^-,t)$}}
\at{-3.8\pscm}{2.3\pscm}{{\footnotesize $h$}}
}
\label{RAD} 
\end{figure*}

A few comments on \thmref{ds} are in order. First, suppose $g:(v,t) \mapsto (Av+a,t)$ is a pure or rational screw translation, with $A^i=I$, so  \eqref{prat} holds. Then \eqref{bbb} yields the uniform bound 
$$
R_h \leq c(\ve) \, i \| a \|. 
$$
This is off by a factor of $c(\ve)$ compared to the optimal inequality of Shimizu-Leutbecher in dimension $2$, generalized by Waterman to higher dimensions. For example, let $g$ be the unit translation $(v,t) \mapsto (v+1,t)$ in $\HH^2$ and suppose $g,h$ generate a torsion-free Fuchsian group. The optimal Margulis constant in this case is known to be $\sinh^{-1}(1)$ \cite{Y}, so the above inequality reduces to  
$$
R_h \leq c(\ve) \leq \frac{1}{\sqrt{2 \cosh(\sinh^{-1}(1))-2}} = \frac{1}{\sqrt{2(\sqrt{2}-1)}} \approx 1.0986, 
$$ 
which should be contrasted with the optimal inequality $R_h \leq 1$. \vs

The real advantage of \eqref{bbb} becomes apparent when $g$ is an irrational screw translation. Waterman's result in $\HH^{n}$ \cite[Theorem 8]{W} is the inequality
\begin{equation}\label{Wat}
R_h \leq K_g \ \| g(v_h)-v_h \|^{1/2} \ \| g(v_{h^{-1}})-v_{h^{-1}} \|^{1/2},
\end{equation}
where $K_g \in [1,2]$ is an explicit constant depending on the rotational part $A$ of $g$. He derives this via a purely algebraic argument based on Clifford matrix representation of M\"{o}bius maps, and assuming $A$ is close enough to the identity. Ignoring this restriction for a moment, replacing $g$ by the iterate $g^i$ in \eqref{Wat}, and using the definition of the $u_{g,i}$ in \eqref{ui}, one obtains
\begin{equation}\label{Wat1}
R_h \leq \frac{2}{c(\ve)} \ \big( u_{g,i}(v_h) \ u_{g,i}(v_{h^{-1}}) \big) ^{1/2} \qquad (i \in \NN).
\end{equation}
Our geometric approach in \thmref{ds}, in effect, replaces each factor on the right side of this inequality with its infimum over $i \in \NN$. Up to a constant, this would be the improved bound \eqref{bbb}. \vs

To get a sense of asymptotics in these inequalities, let us express them in terms of how far the centers $v_h,v_{h^{-1}}$ are from the subspace $E$ on which $g: (v,t) \mapsto (Av+a,t)$ acts as a translation. Take, as before, the decomposition $v=w+w^{\perp}$, so $\| w^\perp \|$ is the Euclidean distance of $v$ to $E$. By \eqref{pyth}, $u_{g,i}(v)$ depends only on $w^{\perp}$. Hence, the quantity  
$$
\tilde{u}_{g,i}(r) = \sup_{\| w^{\perp} \| =r} u_{g,i}(v) 
$$
is well-defined and given by 
$$
\tilde{u}_{g,i}(r) = c(\ve) \ \left( \| A^i-I \|^2 r^2 + i^2 \| a \|^2 \right) ^{1/2}.
$$
In particular,
\begin{equation}\label{YOO}
\tilde{u}_{g,i}(r) = O(r) \qquad \text{as} \ r \to \infty. 
\end{equation}
If we set 
$$
\tilde{\bb}_g(r) = \inf_{i \in \NN} \, \tilde{u}_{g,i}(r),
$$
then $\bb_g(v) \leq \tilde{\bb}_g(r)$ whenever $\| w^{\perp} \| =r$. Since 
$$
\frac{\tilde{\bb}_g(r)}{r} \leq \frac{\tilde{u}_{g,i}(r)}{r} = c(\ve) \, \left( \| A^i-I \|^2 + \frac{i^2 \| a \|^2}{r^2} \right)^ {1/2}, 
$$ 
it follows that 
$$
\limsup_{r \to \infty} \frac{\tilde{\bb}_g(r)}{r} \leq c(\ve) \ \| A^i-I \| \qquad \text{for every} \ i \in \NN. 
$$
By compactness of $\text{SO}(n-1)$, there is a sequence $i_k \to \infty$ for which $A^{i_k} \to I$ as $k \to \infty$. This proves 
\begin{equation}\label{BOO}
\tilde{\bb}_g(r) = o(r) \qquad \text{as} \ r \to \infty. 
\end{equation}
Thus, if $r_h=\|w_h^{\perp} \|$ denotes the distance of $v_h$ to the subspace $E$, Waterman's inequality \eqref{Wat1} together with \eqref{YOO} lead to the bound
\begin{equation}\label{Wat2}
R_h = O(\sqrt{r_h \, r_{h^{-1}}}), 
\end{equation}
while \eqref{bbb} together with \eqref{BOO} give the sharper bound  
\begin{equation}\label{lo}
R_h = o(\sqrt{r_h \, r_{h^{-1}}}) 
\end{equation}
as $r_h,r_{h^{-1}} \to \infty$. 

\section{The $4$-dimensional case}\label{sec:fr}

The preceding analysis takes a much simpler form in $\HH^4$, the space of lowest dimension in which screw translations exist. Let $\Gamma$ be a discrete subgroup of $\text{Isom}^+(\HH^4)$ containing a parabolic isometry which fixes $\infty$. As noted above, the only interesting case for us is when the stabilizer $\Gamma_\infty$ contains no pure or rational screw translations. In this case, if $\Gamma_{\infty}$ is torsion-free, it must be infinite cyclic \cite{EZ}. \vs

We will use the cylindrical coordinates $v=(r,\theta,z)$ on the boundary of $\HH^4$. Pick an irrational screw translation $g$ in $\Gamma_{\infty}$. After a conjugation, we can put $g$ in the normal form 
\begin{equation}\label{gee}
g(r,\theta,z,t)=(r, \theta+2\pi\alpha, z+1,t)
\end{equation}
with $\alpha$ irrational. Thus, the rotational part $A \in \text{SO}(3)$ of $g$ keeps the $z$-axis invariant and rotates in the $(r,\theta)$-plane by an angle $2\pi \alpha$. Note that for each $i \in \NN$ the restriction of $A^i-I$ to the $(r,\theta)$-plane is a conformal linear map, so the functions $u_{g,i}$ and hence the boundary function $\bb_g$ depend only on $r$. It follows that in this case, $\tilde{u}_{g,i}=u_{g,i}$ and $\tilde{\bb}_g=\bb_g$. \vs  

It is easy to see that $r \mapsto \bb_g(r)$ is strictly increasing and piecewise smooth. 
The asymptotic behavior of $\bb_g(r)$ as $r \to \infty$ turns out to depend on the arithmetical properties of $\alpha$ determined by its continued fraction expansion (see \cite{Su} and \cite{EZ}). In a recent work, we have investigated this dependence to prove the following result: 

\begin{theorem}[\cite{EZ}]\label{asym}   
For every irrational $\alpha$, the boundary function $\bb_g$ associated with the parabolic isometry \eqref{gee} satisfies the asymptotically universal upper bound
$$
\bb_g(r) \leq \con \, r^{1/2} \qquad \text{for large} \ r.  
$$ 
If $\alpha$ is Diophantine of exponent $\nu \geq 2$, then $\bb_g$ satisfies the lower bound
$$
\bb_g(r) \geq \con  \, r^{1/(2\nu-2)} \qquad \text{for large} \ r. 
$$
\end{theorem}

\noindent
The upper bound is asymptotically universal in the sense that the constant involved is independent of $\alpha$. In fact, a quantitative version of this result shows that 
$$
\sup_{\alpha} \ \limsup_{r \to \infty} \frac{\bb_g(r)}{\sqrt{r}} < 1,\!000.
$$ 
On the other hand, there are irrationals $\alpha$ of Liouville type for which $\bb_g$ has arbitrarily slow growth over long intervals \cite{EZ}. \vs

The upper bound $\bb_g(r)=O(\sqrt{r})$ of \thmref{asym} is a sharper version of \eqref{BOO}. In the context of our discreteness criterion in \thmref{ds}, it yields the bound  
\begin{equation}\label{fin}
R_h = O (\sqrt[4]{r_h \, r_{h^{-1}}}) 
\end{equation}
as $r_h,r_{h^{-1}} \to \infty$, which is a sharper version of \eqref{lo}. \vs

It is not hard to construct $2$-generator groups of isometries of $\HH^4$ that are demonstrably non-discrete by \eqref{fin}, for which Waterman's criterion \eqref{Wat2} remains inconclusive: In the Cartesian coordinates $(x,y,z)$ of $\RR^3$, let $a=(r,0,0)$ for some $r>0$. Let $\gamma$ be the Euclidean isometry $(x,y,z) \mapsto (-x+2r,y,z)$, and let $\eta$ be the reflection in the sphere $S(a,r^{2/3})$. Define $h=\eta \gamma \in \mob(3)$ and extend it to an isometry of $\HH^4$. Then $h(\infty)=h^{-1}(\infty)=a$, and $h$ maps $S(a,r^{2/3})$ isometrically to itself. Hence $v_h=v_{h^{-1}}=a$, $r_h=r_{h^{-1}}=r$, and $R_h=r^{2/3}$. If $r$ is chosen sufficiently large, the group generated by $h$ and the parabolic $g$ in \eqref{gee} cannot be discrete, for otherwise \eqref{fin} would imply $R_h=O(r^{1/2})$. However, since $R_h = O(r)$, this non-discreteness is not detected by \eqref{Wat2}.

\end{document}